\documentclass[11pt,reqno]{amsart}
\usepackage{amsmath,amsthm,amssymb,amscd,url,enumerate,mathtools}
\usepackage{graphicx}
\usepackage[margin=1in]{geometry}
\usepackage{afterpage}
\usepackage{tikz}
\usepackage{todonotes}
\usepackage[all]{xy}
\usepackage[colorlinks=true,citecolor={blue},linkcolor = {purple}]{hyperref} 
\usepackage{tabularx}
\usepackage{tabu}

\usepackage{dutchcal}

\usepackage[square]{natbib}
\setcitestyle{numbers}

\usepackage{adjustbox} 

\newcommand{\TITLE}{Non-monogenic Division Fields of Elliptic Curves}
\newcommand{\TITLERUNNING}{}

\theoremstyle{plain}
\newtheorem{theorem}{Theorem}

\newtheorem{lemma}[theorem]{Lemma}
\newtheorem{corollary}[theorem]{Corollary}

\theoremstyle{definition}

\theoremstyle{remark}
\newtheorem{remark}[theorem]{Remark}

\newtheorem{example}[theorem]{Example}
\newtheorem{question}[theorem]{Question}

\numberwithin{theorem}{section}

  {\end{list}}

  {\end{list}}

\newcommand{\tightoverset}[2]{%
  \mathop{#2}\limits^{\vbox to -.5ex{\kern-1.05ex\hbox{$#1$}\vss}}}

\newcommand{\gp}{{\mathfrak{p}}}

\def\Ocal{{\mathcal O}}

\newcommand{\FF}{\mathbb{F}}
\newcommand{\GG}{\mathbb{G}}

\newcommand{\QQ}{\mathbb{Q}}

\newcommand{\ZZ}{\mathbb{Z}}

\newcommand{\tensor}{\otimes}

\newcommand{\ol}[1]{\overline{#1}}

\newcommand{\Aut}{\operatorname{Aut}}
\newcommand{\GL}{\operatorname{GL}}
\newcommand{\Gal}{\operatorname{Gal}}

\newcommand{\End}{\operatorname{End}}
\newcommand{\irred}{\operatorname{irred}}
\newcommand{\ord}{\operatorname{ord}}

\newcommand{\tors}{\operatorname{tors}}

\newcommand*{\Scale}[2][4]{\scalebox{#1}{#2}}

\title[\TITLERUNNING]{\vspace*{-1.3cm} \TITLE}

\date{\today}

\author[Hanson Smith]{Hanson Smith}
\address{Department of Mathematics, University of Connecticut,
341 Mansfield Road U1009
Storrs, CT 06269-1009
USA}
\email{hanson.smith@uconn.edu or hansonsmith101@gmail.com}


\thanks{The author would like to thank \'Alvaro Lozano-Robledo for the question that inspired this project and the idea for finding explicit families. This research was partially supported by NSF-CAREER CNS-1652238 under the supervision of PI Dr. Katherine E. Stange.}

\begin{document}
\sloppy

\begin{abstract}
For various positive integers $n$, we show the existence of infinite families of elliptic curves over $\mathbb{Q}$ with $n$-division fields, $\mathbb{Q}(E[n])$, that are {\bf not} monogenic, i.e., the ring of integers does not admit a power integral basis. We parametrize some of these families explicitly. Moreover, we show that every $E/\mathbb{Q}$ without CM has infinitely many non-monogenic division fields. Our main technique combines a global description of the Frobenius obtained by Duke and T\'oth with a simple algorithm based on ideas of Dedekind.
\end{abstract}

\maketitle

\section{Introduction}

Denote a primitive $n^{\text{th}}$ root of unity by $\zeta_n$ and take the cyclotomic field $\QQ(\zeta_n)$. The arithmetic of these fields has been the inspiration for over a century of algebraic number theory. For example, Kummer first developed the notion of `ideal numbers' in cyclotomic fields; see \cite{KummerIdeal1} and \cite[pages 1-47]{KummerIdeal2}. Implicit in these investigations is the fact that the ring of integers of $\QQ(\zeta_n)$ is $\ZZ\left[\zeta_n\right]$. To put it another way, the $\GG_m$ division fields are monogenic. For a nice survey of the literature related to the above discussion see \cite[Chapter 4.4]{Nark}. 

Naturally, we want to know about the arithmetic of division fields of other algebraic groups. In this paper we will study the monogeneity of division fields\footnote{We will use `division field' and `torsion field' interchangeably to indicate an extension of the form $\QQ\left(E[n]\right)$.} of elliptic curves. 
In \cite{abdiv}, Lozano-Robledo and Gonz\'alez-Jim\'enez classify the possible abelian division fields. In the process, they show that for $n=1,2,3,4$, and 5 one can have $\QQ\left(E[n]\right)=\QQ\left(\zeta_n\right)$. Hence, in these cases, $\QQ\left(E[n]\right)$ is monogenic. The monogeneity of a family of partial 3-torsion fields is studied in \cite{GassertSmithStange}. Adelmann \cite{Adelmann} has an in-depth investigation of the splitting of primes in torsion fields. Cassou-Nogu\`es and Taylor have studied the monogeneity of torsion fields of CM elliptic curves extensively; see \cite{C-NT}, \cite{C-NTModUnits}, and \cite{C-NTMono}.

The outline of the paper is as follows. In Section \ref{EndFrob}, we establish which imaginary quadratic orders we need to consider, and we state a theorem of Duke and T\'oth \cite{DukeToth} that describes the action of any lift of the Frobenius at $p$ on prime-to-$p$-torsion points of an elliptic curve. This theorem will be essential for our current work; however, in a forthcoming investigation into similar questions regarding abelian varieties of dimension greater than one, we will generalize Centeleghe's description of the Frobenius in \cite{Centeleghe}. In Section \ref{DedCrit}, we review Dedekind's criterion for non-monogeneity. With Section \ref{NonMonoTors}, we establish criteria for certain torsion fields of elliptic curves to be non-monogenic and remark that a positive proportion of elliptic curves fall into these families. We then use these ideas to show that every elliptic curve $E/\QQ$ without CM has infinitely many torsion fields that are not monogenic. Section \ref{Families} establishes explicit families in terms of either one or two parameters. Lastly, Section \ref{Tables} includes tables for Theorem \ref{maintables} for $p=3,5,7,$ and 11; the table for $p=2$ is included in Section \ref{NonMonoTors}.

\section{Describing Endomorphism Rings and the Frobenius}\label{EndFrob}
Let $p$ be a rational prime, $E$ an elliptic curve over $\FF_p$, and $a_p$ the trace of Frobenius of $E$. From Hasse we have the bound $|a_p|\leq 2\sqrt{p}.$ 
Denote the Frobenius endomorphism of $E$ by $\pi$. 
Recall, $\pi$ satisfies $f_{\pi}=x^2-a_px+p.$
The discriminant of the Weil polynomial $f_{\pi}$ is $\Delta_{\pi}:=a_p^2-4p,$
so we obtain $4p=a_p^2-\Delta_{\pi}.$ 

We would like to know the possibilities for the endomorphism ring of $E$ over $\FF_p$, which we denote $\End_{\FF_p}(E)$. If $E$ is ordinary, i.e., $a_p\not\equiv 0\bmod p$, then \cite[Theorem 4.2]{Waterhouse} shows $\End_{\FF_p}(E)$ can be any order in the imaginary quadratic field $\QQ(\pi)$ containing $\ZZ[\pi]$.

If $E$ is supersingular, $\End_{\ol{\FF_p}}(E)$ is an order in a quaternion algebra. However, over $\FF_p$, we have $\End_{\FF_p}(E)\tensor\QQ\cong \QQ(\pi)$; see \cite[Theorem 8]{WaterhouseMilne}. Theorem 4.2 of \cite{Waterhouse} combined with some computation shows that all possible endomorphism rings occur. Specifically, if $p=2$, then $\End_{\FF_p}(E)\cong \ZZ\left[\sqrt{-2}\right]$ when $a_2=0$ and $\End_{\FF_p}(E)\cong\ZZ\left[\sqrt{-1}\right]$ when $a_2=\pm 2$. If $p=3$, we have $\End_{\FF_p}(E)\cong\ZZ\left[\sqrt{-3}\right]$ or $\ZZ\left[\frac{1+\sqrt{-3}}{2}\right]$ 
for all supersingular $a_3$, i.e., $0,\pm 3$. Otherwise, $\End_{\FF_p}(E)\cong\ZZ\left[\sqrt{-p}\right]$ if $p\equiv 1 \bmod 4$ and $\End_{\FF_p}(E)\cong\ZZ\left[\sqrt{-p}\right]$ or $\ZZ\left[\frac{1+\sqrt{-p}}{2}\right]$ if $p\equiv 3\bmod 4$.

We will also make use of an explicit description of the Frobenius from Duke and T\'oth \cite{DukeToth}. For our purposes, we restrict their theorem to the case where the finite field is $\FF_p$. Henceforth, $\End(E):=\End_{\FF_p}(E)$ and $\Delta_{\End}$ will denote the discriminant of $\End(E)$. Define $b_p$ to be the unique positive integer such that $4p=a_p^2-\Delta_{\End} b_p^2.$ 
In other words, $b_p$ is the index of the order generated by the Frobenius in $\End(E)$, so $\Delta_{\pi}=b_p^2\Delta_{\End}.$ 
Consider now the matrix
\begin{equation}\label{Frob}
\sigma_p = \begin{bmatrix} 
\dfrac{a_p+b_p\delta_{\End}}{2} & b_p \\
\dfrac{b_p(\Delta_{\End}-\delta_{\End})}{4} & \dfrac{a_p-b_p\delta_{\End}}{2} 
\end{bmatrix},
\end{equation}
where $\delta_{\End} =0,1$ according to whether $\Delta_{\End}\equiv 0,1$ modulo 4.
\begin{theorem}\label{DandT} \cite[Theorem 2.1]{DukeToth} Let $\hat{E}/\QQ$ be an elliptic curve whose reduction at $p$ is $E$. For $n$ prime to $p$, write $\QQ\left(\hat{E}[n]\right)$ for the $n^{\text{th}}$ torsion field. Then $p$ is unramified in $\QQ\left(\hat{E}[n]\right)$ and the integral matrix $\sigma_p$, when reduced modulo $n$, represents the class of the Frobenius at $p$ in $\Gal\left(\QQ\left(\hat{E}[n]\right)/\QQ\right)$.
\end{theorem}

In other words, $\sigma_{p}$ yields the action of the Frobenius on the $l$-adic Tate module $T_l\left( \hat{E}\right)$ for any $l$ that is prime to $p$. We note that the above theorem holds with the appropriate modifications for elliptic curves over arbitrary number fields; however, we restrict to $\QQ$ since it simplifies the exposition and suffices for our investigation.

\section{Dedekind's Criterion for Non-monogeneity}\label{DedekindSection}
We start with Dedekind's criterion, first appearing in \cite{Dedekind}, for the splitting of primes in number fields:

\begin{theorem}\label{DedCrit}
Let $f(x)\in \ZZ[x]$ be monic and irreducible with $\alpha$ a root. Write $K:=\QQ(\alpha)$ with $\Ocal_K$ denoting the ring of integers. If $p\in \ZZ$ is a prime that does not divide $[\Ocal_K:\ZZ[\alpha]]$, then the factorization of $p$ in $\Ocal_K$ coincides with the factorization of $f(x)$ modulo $p$. That is, if
\[f(x)\equiv \phi_1(x)^{e_1}\cdots\phi_k(x)^{e_k}  \bmod p\]
is a factorization of $f(x)$ into irreducibles in $\ZZ/p\ZZ[x]$, then $p$ factors into primes in $\Ocal_K$ as 
\[p=\gp_1^{e_1}\cdots \gp_k^{e_k}.\]
Further, the residue degree of $\gp_i$ is equal to the degree of $\phi_i$.
\end{theorem}

Many texts on algebraic number theory contain a proof; e.g., \cite[Theorem 4.33]{Nark}. For a nice expository proof see Keith Conrad's note \cite{KCDedekind}.  


Using this criterion, Dedekind was the first to demonstrate a number field that was not monogenic. Dedekind considered the cubic field generated by a root of $x^3-x^2-2x-8$. He showed that the prime 2 split completely. If there was a possible power integral basis, then one would have to find a cubic polynomial generating the same number field that splits completely into distinct linear factors modulo 2. Since there are only two distinct linear polynomials in $\FF_2[x]$, this is impossible. Hence the number field cannot be monogenic. We will use the same strategy as Dedekind to construct non-monogenic torsion fields. 

In fact, Dedekind showed something stronger than non-monogeneity. The methods outlined above show that $2$ will divide the index of any monogenic order in the maximal order. Such a prime is called an \emph{essential discriminant divisor}.\footnote{The terms \emph{common index divisor} and, confusingly, \emph{inessential discriminant divisor} also appear in the literature.} Hensel \cite{Hensel1894} showed that the essential discriminant divisors are exactly the primes whose splitting is too much for the number of irreducible factors of the correct degree to accommodate. 

\section{Non-monogenic torsion fields}\label{NonMonoTors}

Our goal is to find torsion fields where a prime $p$ splits into a large number of primes, but there are fewer irreducible polynomials of the correct degree in $\FF_p[x]$. In other words, we would like to find when the prime $p$ is an obstruction to monogeneity. 
As a biproduct of our methods, we will obtain an algorithm for computing the unramified essential discriminant divisors of $\QQ(E[n])$ and a proof that an elliptic curve over $\QQ$ without CM has infinitely many non-monogenic division fields. 

It will be convenient to denote the order of the matrix $\sigma_p$ modulo $n$, i.e., in $\GL_2\left(\ZZ/n\ZZ\right)$, by $\ord\left(\sigma_p,n\right)$. We are going to compute $\ord\left(\sigma_p,n\right)$ for a number of $n$ that are relatively prime to $p$ and then compare this to $[\QQ(E[n]):\QQ]$. Finding $[\QQ(E[n]):\QQ]$ is a difficult problem in general, so here we will focus on what we expect to be the generic case: 
\[[\QQ(E[n]):\QQ]=|\GL_2(\ZZ/n\ZZ)|.\]
The following corollary of Han will be useful for computing $|\GL_2(\ZZ/n\ZZ)|$:

\begin{lemma}[Corollary 2.8, \cite{HanordGL}]\label{orderofGL} Let $n>1$ be an integer and $n=p_1^{e_1}\cdots p_k^{e_k}$ its prime factorization. Then, 
\[|\GL_2(\ZZ/n\ZZ)|=\prod_{i=1}^k|\GL_2\left(\ZZ/p_i^{e_i}\ZZ\right)|=\prod_{i=1}^k p_i^{4\left(e_i-1\right)}\left(p_i^2-1\right)\left(p_i^2-p_i\right).\]
\end{lemma}


We will also need to know the number of irreducible polynomials in $\FF_p[x]$ of degree equal to $\ord\left(\sigma_p,n\right)$. The goal is to demonstrate that there are not enough irreducible polynomials of the proper degree, so Dedekind's criterion cannot hold. For this we will need the following result of Gauss \cite{GaussHigher}.

\begin{lemma}\label{irredpolys}
Let $\mu$ be the M\"obius function. The number of irreducible polynomials of degree $m$ in $\FF_p[x]$ is given by
\[\irred(m,p):=\dfrac{1}{m}\sum_{d\mid m}p^d\mu\left(\dfrac{m}{d}\right).\]
\end{lemma}

One can see \cite{chebmincounting} for a modern proof.

Using SageMath \cite{Sage}, we can compute $\ord\left(\sigma_p,n\right)$, $\irred\left(\ord\left(\sigma_p,n\right),p\right)$, and $\frac{|\GL_2(\ZZ/n\ZZ)|}{\ord\left(\sigma_p,n\right)}$ for $n$ less than $1000$. Here $\frac{|\GL_2(\ZZ/n\ZZ)|}{\ord\left(\sigma_p,n\right)}$ is the number of primes $p$ splits into in $\QQ(E[n])$. If 
\[\irred(\ord\left(\sigma_p,n\right),p)<\frac{\left|\GL_2(\ZZ/n\ZZ)\right|}{\ord\left(\sigma_p,n\right)},\] 
then $p$ must divide the index of any monogenic order in $\Ocal_{\QQ(E[n])}$.

\begin{example}\label{p=2n=11} Take $p=2$. Over $\FF_2$ the possible ordinary traces of Frobenius are $\pm 1$. Hence the discriminant of the characteristic polynomial of Frobenius is
\[\Delta_{\pi}=1-8=-7.\]
Since $-7$ is square-free, we see $b_2=1$. For example, when $a_2=1$, the matrix representing the class of Frobenius is 
\[\sigma_2 = \begin{bmatrix} 
8/2 & (-7\cdot 8)/4 \\
1 & -6/2 
\end{bmatrix}=
\begin{bmatrix} 
4 & -14 \\
1 & -3 
\end{bmatrix}.\]
We find that $\sigma_2$ has order $10$ modulo $11$. Assuming $[\QQ(E[11]):\QQ]=|\GL_2(\ZZ/11\ZZ)|$, we compute that $2$ splits into $1320$ primes in $\QQ(E[11])$. However, there are only $99$ irreducible polynomials of degree 10 in $\FF_2[x]$. Hence, if $E$ is an elliptic curve over $\QQ$ such that the mod 11 Galois representation is surjective and with $a_2=1$, then $\QQ(E[11])$ is not monogenic. In particular, 2 will be an essential discriminant divisor for $\QQ(E[11])/\QQ$.

\end{example}

\begin{theorem}\label{maintables}
Let $p$, $a_p$, $b_p$, and $n$ be as in Table \ref{Frobenius2}, \ref{Frobenius3}, \ref{Frobenius5}, \ref{Frobenius7}, or \ref{Frobenius11}. If $E$ is an elliptic curve over $\QQ$ whose reduction at $p$ has trace of Frobenius $a_p$, index $b_p$, and such that the Galois representation
\[\rho_{E,n}:\Gal(\QQ(E[n])/\QQ)\to \GL_2(\ZZ/n\ZZ)\]
is surjective, then, for the listed $n$, the torsion field $\QQ(E[n])$ is \textbf{not} monogenic. Moreover, $p$ is an essential discriminant divisor.

If, instead of asking $\rho_{E,n}$ to be surjective, we require that $E$ is a Serre curve, then $\rho_{E,n}$ may have index 2 in $\GL_2(\ZZ/n\ZZ)$. 
For the {\color{red} red $n$}, when $\left[\GL_2(\ZZ/n\ZZ):\rho_{E,n}\left(\Gal(\QQ(E[n])/\QQ)\right)\right]=2$, the prime $p$ is no longer an essential discriminant divisor.
\end{theorem}

Note that employing the process outlined above for all primes up to $[\QQ(E[n]):\QQ]$, will identify all essential discriminant divisors. Thus if a positive integer $n$ less than $1000$ is not listed in one of the tables below, then the primes 2, 3, 5, 7, and 11 are not essential discriminant divisors of $\QQ(E[n])/\QQ$; however, even if there are no essential discriminant divisors it may still be the case that $\QQ(E[n])$ is not monogenic over $\QQ$.

\begin{table}[h!]
\centering
\begin{tabular}{ | m{.5cm} | m{2.5cm}| m{8.5cm} | }
\hline
$a_2$ & $\sigma_2$ & non-monogenic $n$ \\
\hline
0 &  \begin{center} \Scale[.9]{$\begin{bmatrix} 
0 & 1 \\
2 & 0
\end{bmatrix}$}\end{center} & 3, {\color{red}5,} 9, 11, 15, {\color{red}17}, 21, {\color{red}27,} 33, 43, 51, 57, 63, 85, 91, 93, 105, 117,
129, 171, 195, 255, {\color{red}257,} 273, 315, 331, 341, 381, 455, 513, 585, {\color{red}657,}
683, 771, 819, 993  \\ [.1ex]
\hline
1 &  \begin{center} \Scale[.9]{$\begin{bmatrix} 
1 & 1 \\
-2 & 0
\end{bmatrix}$}\end{center} & 11  \\ [.1ex]
\hline
-1  & \begin{center}\Scale[.9]{$\begin{bmatrix} 
0 & 1 \\
-2 & -1
\end{bmatrix}$} \end{center} & 11, 23  \\ [.1ex]
\hline
2  & \begin{center} \Scale[.9]{$\begin{bmatrix} 
1 & 1 \\
-1 & 1
\end{bmatrix}$} \end{center} & 5, 13, 15, {\color{red}17,} 41, 51, 65, 85, 91, 105, 117, {\color{red}145,} 195, 205, 255, {\color{red}257,}
273, 315, 455, 565, 585, 771, 819  \\
\hline
 $-2$  & \begin{center} \Scale[.9]{$\begin{bmatrix} 
-1 & 1 \\
-1 & -1 
\end{bmatrix}$} \end{center} & 5, 13, 15, {\color{red}17,} 41, 51, 65, 85, 91, 105, 117, {\color{red}145,} 195, 205, 255, {\color{red}257,}
273, 315, 455, 565, 585, 771, 819 \\
\hline
\end{tabular}
\vspace{.05 in}
\caption{Using the splitting of 2 in $\QQ(E[n])$ to show non-monogeneity for $n<1000$.}
\label{Frobenius2}
\end{table}

\begin{remark}\label{extensions} 
One can generalize the above algorithm to allow for elliptic curves over arbitrary number fields. If $K$ is a number field in which there is a prime above $p$ with residue class degree 1 and \[[K(E[n]):K]=|\GL_2(\ZZ/n\ZZ)|,\] then our tables still identify non-monogenic torsion fields. For the relevant generalization of Dedekind's criterion see \cite[Chapter 1, Theorem 7.4]{Janusz}.

When the residue class degree is not 1 or the degree of the $n$-torsion field is smaller, the mechanics of our algorithm still work. However, there will be fewer obstruction to monogeneity. Increasing the residue class degree yields more irreducible polynomials of the correct degree, and a smaller $n$-torsion field will lead to less splitting.
\end{remark}


Each of the $n$ that are not red in the tables corresponds to infinitely many elliptic curves, in fact a positive proportion of elliptic curves. To see this, we summarize a theorem of Jones:
\begin{theorem}\label{AlmostAll}
\cite[Theorem 4]{Jones} Almost all elliptic curves are Serre curves.
\end{theorem}
Jones counts elliptic curves by na\"ive height, which works amicably with our methods for for $p\geq 5$. For $p=2$ or $3$, we note that a given trace of Frobenius corresponds to a positive density of elliptic curves. Hence almost all elliptic curves with that trace of Frobenius are Serre curves. Thus every $n$ in Tables \ref{Frobenius2}, \ref{Frobenius3}, \ref{Frobenius5}, \ref{Frobenius7}, and \ref{Frobenius11} that is not red corresponds to infinitely many elliptic curves. The interested reader should note that Radhakrishnan \cite{Radhakrishnan} computed an asymptotic formula for the number of non-Serre curves in his thesis.


Glancing over our tables, it appears that supersingular primes are often obstructions to monogeneity. We can make this precise in a certain case:

\begin{theorem}\label{SupersingularVertically}
Let $p>3$ be a prime for which $E$ has good supersingular reduction and suppose $[\QQ(E[p+1]):\QQ]>p^2-p,$ 
then $\QQ(E[p+1])$ is not monogenic over $\QQ$. Specifically, $p$ is an essential discriminant divisor of $\QQ(E[p+1])$ over $\QQ$.
\end{theorem}

\begin{proof}
Since $E$ is supersingular at $p>3$, we know $a_p=0$. Thus the matrix representing the class of Frobenius is either
\[\begin{bmatrix} 
0 & 1 \\
-p & 0
\end{bmatrix} \ \ \  \text{ or } \ \ \ 
\begin{bmatrix} 
1 & 2 \\
\frac{-p-1}{2} & -1
\end{bmatrix},\]
depending on whether $b_p=1$ or 2. We compute that both these matrices have order 2 modulo $p+1$. Thus there are more than $\frac{p^2-p}{2}$ 
primes above $p$ in $\QQ(E[p+1])$ and all of these primes above $p$ have residue degree 2. Using Lemma \ref{irredpolys}, there are exactly $\frac{p^2-p}{2}$ irreducible polynomials of degree 2 in $\FF_p[x]$. 
Thus the splitting of $p$ cannot be accommodated by irreducible polynomials of degree 2 in $\FF_p[x]$ and our result follows.
\end{proof}

Using the estimate
\[|\GL_2(\ZZ/(p+1)\ZZ)|\geq (p+1)^4-\frac{1}{2}(p+1)^4-\frac{1}{4}(p+1)^4+\frac{1}{8}(p+1)^4=\frac{3}{8}(p+1)^4,\] we expect $[\QQ(E[p+1]):\QQ]$ to often be much larger than $p^2-p$. In particular, Theorem \ref{SupersingularVertically} holds for Serre curves. 

Combining the ideas above we see a given elliptic curve over $\QQ$ without CM has infinitely many non-monogenic division fields.

\begin{corollary}\label{InfiniteNonMono}
Let $E/\QQ$ be an elliptic curve without CM, then for infinitely many $n>1$ the division field $\QQ(E[n])$ is not monogenic.
\end{corollary}

\begin{proof}
Serre's open image theorem \cite{SerreOpenImage} shows that the image of 
\[\rho_E:\Gal\left(\ol{\QQ}/\QQ\right)\to \Aut\left(E_{\tors}\right)\cong \GL_2\left(\hat{\ZZ}\right)\] is open and hence has finite index. Write $I:=\left[\GL_2\left(\hat{\ZZ}\right):\rho_E\left(\Gal\left(\ol{\QQ}/\QQ\right)\right)\right]$. 

Note that all sufficiently large primes $p$ satisfy 
\begin{equation}\label{IInequality}
\frac{3}{16I}(p+1)^4>p^2-p.
\end{equation} 
Elkies \cite{Elkies} has shown that $E$ has infinitely many supersingular primes. Hence there are infinitely many supersingular primes satisfying \eqref{IInequality}. For these primes, the splitting of $p$ in $\QQ(E[p+1])$ cannot be mirrored by the factorization of a polynomial in $\FF_p[x]$. Thus, the division field $\QQ(E[p+1])$ is not monogenic over $\QQ$. 
\end{proof}

\section{Explicit Families}\label{Families}

We would like to construct explicit families of elliptic curves with non-monogenic torsion fields.
Daniels \cite{Daniels} has shown the family
\begin{equation}\label{DanielsFamily}
E_t: \ \ y^2+xy=x^3+t,
\end{equation}
with $t\in \QQ$, consists of Serre curves with at most 12 exceptions. Using the methods outlined about, we can quickly compute that if $t$ is odd, then the trace of Frobenius at $p=2$ is $-1$. Hence $\QQ(E_t[11])$ and $\QQ(E_t[23])$ are not monogenic with the exception of at most 12 values of $t$. Note that if $t$ is even, then $E_t$ is singular at 2. There is nothing particularly special about 2 here and we can do the same thing for other small primes. For example, if $t\equiv 2 \bmod 7$, then $E_t$ is supersingular at 7 and for $n$ equal to 4, 8, 16, 24, 43, 48, 75, 80, 86, 96, 100, 120, 150, 160, 172, 191, 200, 240, 300, 344, 382, 400, 480, 600, 684, 688, 764, 774, 800, 817, or 912, the division field $\QQ(E[n])$ is not monogenic over $\QQ$, again with the exception of at most 12 values of $t$.

For another family, we require a result of Mazur. 
\begin{theorem}[Theorem 4 of \cite{MazurRatIso}]\label{MazurSemistable}
Let $E/\QQ$ be a semistable elliptic curve. For every prime $p\geq 11$, the Galois representation
\[\rho_{E,p}:E[p]\to \GL_2(\ZZ/p\ZZ)\]
is surjective. 
\end{theorem}

Recall that for an elliptic curve $E$, if the standard coefficient of the short Weierstrass form $c_4$ and the discriminant $\Delta$ are relatively prime, then $E$ is semistable. See \cite[VII.5]{AEC}. We consider an elliptic curve over $\FF_p$ with a given trace of Frobenius. For example, the elliptic curve $y^2+y=x^3+x^2+1$ over $\FF_2$ has trace of Frobenius $a_2=2$. Note that $c_4$ is $16$ and does not depend on the coefficient $a_6$. Hence, the family of rational elliptic curves
\[E_s: \ \ y^2+y=x^3+x^2+s\]
also has $c_4=16$. One computes the discriminant is $\Delta_s=-432 s^{2} - 280 s - 43$. Since $\Delta_s$ is always odd, $\gcd\left(\Delta_s,c_4\right)=1$. Therefore, $E_s$ is a family of semistable elliptic curves. Using Theorem \ref{MazurSemistable} in conjunction with our methods outlined above, we find that every curve in the family $E_s$ has non-monogenic $13$ and $41$ division fields.

Similarly to the above,
\begin{equation*}
E_{u,v}: y^2 +uy = x^3+vx^2
\end{equation*} 
has discriminant $u^2\left(16v^3+27u^2\right)$ and the coefficient $c_4$ is equal to $16 v^2$. If we require that $u$ is odd, $v$ is even, and $\gcd(3u,v)=1$, then $E_{u,v}$ is a family of semistable elliptic curves that has trace of Frobenius $a_2=0$. Hence the 11, 43, 331, and 683 division fields are not monogenic over $\QQ$. 

If one is interested in a given one-parameter family, work of Cojocaru, Grant, and Jones \cite{CojocaruGrantJones}, shows that in any one-parameter family with no obstructions almost all elliptic curves are Serre curves.

\section{Further Questions}

Two natural questions that arise after seeing that there are indeed many division fields that are not monogenic over $\QQ$ are:

\begin{question}\label{when?}
Precisely when is a given torsion field $\QQ(E[n])$ monogenic over $\QQ$?
\end{question}

\begin{question}\label{where?}
Given an elliptic curve $E/K$ and a torsion field $K(E[n])$, are there subfields $L\subsetneq K(E[n])$ over which $K(E[n])$ is monogenic?
\end{question}

Both Questions \ref{when?} and \ref{where?} are excellent motivation, but complete answers are likely somewhat far away. 

Because of the Weil pairing, a natural $L$ to consider for Question \ref{where?} is $K(\zeta_n)$. Looking at $\QQ(E[2])$ over $\QQ$ should be a nice place to test this hypothesis. Computationally, $\QQ(E[2])$ is not often monogenic over $\QQ$. 
However, there seem to be special families with monogenic 2-torsion fields. For example, Legendre form, $E_{\lambda}: \ y^2=x(x-1)(x-\lambda),$ requires solving 
\[g_{\lambda}: \ \lambda^{6} - 3 \lambda^{5} + \left(6- \frac{j}{256} \right) \lambda^{4} + \left(\frac{j}{128} -7\right) \lambda^{3}
 +\left(6- \frac{j}{256} \right) \lambda^{2} - 3 \lambda + 1,\]
where $j$ is the $j$-invariant of $E_{\lambda}$. When $j\in \ZZ$ is divisible by 256 and $\frac{j}{64}-27$ is square-free, then $\lambda$ seems to yield a monogenic generator of $\QQ\left(E_{\lambda}[2]\right)$ over $\QQ$.


\section{Tables for $p=3,5,7,$ and 11}\label{Tables}

For $p\geq 3$, we have possibilities for $b_p$ other than $b_p=1$. This means that, over $\FF_p$ with $p\geq 3$, we have elliptic curves with endomorphism rings that are not generated by the Frobenius endomorphism.

\begin{table}[h!]
\centering
\begin{adjustbox}{width=
5.3 in,center}
\begin{tabular}{ | m{.5cm} | m{.5cm} | m{2.5cm}| m{9.37cm} | }
\hline
$a_3$ & $b_3$ & $\sigma_3$ & non-monogenic $n$ \\
\hline
0 & 1 & \begin{center} \Scale[.9]{$\begin{bmatrix} 
0 & 1 \\
-3 & 0
\end{bmatrix}$} \end{center} &  
4, 7, 8, 14, 16, 20, 28, 40, 52, 56, 61, 80, 91, 104, 122, 160, 164,
182, 205, 244, 259, 266, 328, 364, 410, 484, 488, 518, 532, 547, 656,
661, 671, 703, 728, 820, 949, 968 \\ [.1ex]
\hline
0 & 2 & \begin{center} \Scale[.9]{$\begin{bmatrix} 
1 & 2 \\
-2 & -1
\end{bmatrix}$} \end{center} & 
{\color{red}2,} 4, 7, 8, 14, 16, 20, 28, 40, 52, 56, 61, 80, 91, 104, 122, 160, 164,
182, 205, 244, 259, 266, 328, 364, 410, 484, 488, 518, 532, 547, 656,
661, 671, 703, 728, 820, 949, 968  \\ [.1ex]
\hline
1 & 1 & \begin{center} \Scale[.9]{$\begin{bmatrix} 
1 & 1 \\
-3 & 0
\end{bmatrix}$} \end{center} & 5{\color{red}, 40}  \\ [.1ex]
\hline
-1  & 1 & \begin{center}\Scale[.9]{$\begin{bmatrix} 
0 & 1 \\
-3 & -1
\end{bmatrix}$} \end{center} & 5{\color{red}, 23, 40}  \\ [.1ex]
\hline
2  & 1 & \begin{center} \Scale[.9]{$\begin{bmatrix} 
1 & 1 \\
-2 & 1
\end{bmatrix}$} \end{center} & {\color{red}4,} 11, {\color{red}22,} 136, 272  \\ [.1ex]
\hline
 $-2$  & 1 & \begin{center} \Scale[.9]{$\begin{bmatrix} 
-1 & 1 \\
-2 & -1 
\end{bmatrix}$} \end{center} & {\color{red}4, 22,} 136, 272 \\ [.1ex]
\hline
 $3$  & 1 & \begin{center} \Scale[.9]{$\begin{bmatrix} 
2 & 1 \\
-1 & 1
\end{bmatrix}$} \end{center} & 7, 14, 28, 52, 56, 91, 104, 182, 259, 266, 364, 518, 532, 703, 728,
949 \\ [.1ex]
\hline
 $-3$  & 1 & \begin{center} \Scale[.9]{$\begin{bmatrix} 
-1 & 1 \\
-1 & -2
\end{bmatrix}$} \end{center} & 7, 14, 28, 52, 56, 91, 104, 182, 259, 266, 364, 518, 532, 703, 728,
949 \\ [.1ex]
\hline
\end{tabular}
\end{adjustbox}
\vspace{.05 in}
\caption{Using the splitting of 3 in $\QQ(E[n])$ to show non-monogeneity for $n<1000$.}
\label{Frobenius3}
\end{table}

\newpage

\begin{table}[h!]
\centering
\begin{adjustbox}{width=6 in,center}
\begin{tabular}{ | m{.5cm} | m{.5cm}| m{2.5cm}| m{8.6cm} | }
\hline
$a_5$ & $b_5$ & $\sigma_5$ & non-monogenic $n$ \\
\hline
0 & 1 & \begin{center} \Scale[.9]{$\begin{bmatrix} 
0 & 1 \\
-5 & 0 
\end{bmatrix}$} \end{center} & 
3, 6, 8, 12, {\color{red}18,} 21, 24, 39, 42, 48, 52, 63, 78, 104, 126, 156, {\color{red}186,}
208, 217, 248, 252, 279, 312, 372, 434, 449, 504, 521, 558, 624, 651,
744, 868, 898, 939  \\ [.1 ex]
\hline
1 & 1 & \begin{center} \Scale[.9]{$\begin{bmatrix} 
1 & 1 \\
-5 & 0 
\end{bmatrix}$} \end{center} & 11, 28, 56  \\ [.1 ex]
\hline
$-1$  & 1 & \begin{center}\Scale[.9]{$\begin{bmatrix} 
0 & 1 \\
-5 & -1 
\end{bmatrix}$} \end{center} & 28, 56  \\ [.1 ex]
\hline
2  & 1 & \begin{center} \Scale[.9]{$\begin{bmatrix} 
1 & 1 \\
-4 & 1
\end{bmatrix}$} \end{center} &   \\ [.1 ex]
\hline
 $-2$  & 1 & \begin{center} \Scale[.9]{$\begin{bmatrix} 
-1 & 1 \\
-4 & -1
\end{bmatrix}$} \end{center} &  \\ [.1 ex]
\hline
 $2$  & 2 & \begin{center} \Scale[.9]{$\begin{bmatrix} 
1 & 2 \\
-2 & 1
\end{bmatrix}$} \end{center} & {\color{red}2,} 4, 8, 48 \\ [.1 ex]
\hline
 $-2$  & 2 & \begin{center} \Scale[.9]{$\begin{bmatrix} 
-1 & 2 \\
-2 & -1
\end{bmatrix}$} \end{center} & {\color{red}2,} 4, 8, 48 \\ [.1 ex]
\hline
 $3$  & 1 & \begin{center} \Scale[.9]{$\begin{bmatrix} 
2 & 1 \\
-3 & 1
\end{bmatrix}$} \end{center} & 3, {\color{red}18,} 24, 36, 72 \\ [.1 ex]
\hline
 $-3$  & 1 & \begin{center} \Scale[.9]{$\begin{bmatrix} 
-1 & 1 \\
-3 & -2
\end{bmatrix}$} \end{center} & 3, {\color{red}18,} 24, 36, 72 \\ [.1 ex]
\hline
 $4$  & 1 & \begin{center} \Scale[.9]{$\begin{bmatrix} 
2 & 1 \\
-1 & 2
\end{bmatrix}$} \end{center} & 8, 48 \\ [.1 ex]
\hline
 $-4$  & 1 & \begin{center} \Scale[.9]{$\begin{bmatrix} 
-2 & 1 \\
-1 & -2 
\end{bmatrix}$} \end{center} & 8, 48 \\ [.1 ex]
\hline
\end{tabular}
\end{adjustbox}
\vspace{.05 in}
\caption{Using the splitting of 5 in $\QQ(E[n])$ to show non-monogeneity for $n<1000$.}
\label{Frobenius5}
\end{table}

\newpage

\begin{table}[h!]
\centering
\begin{adjustbox}{width=4.9 in,center}
\begin{tabular}{ | m{.5cm} | m{.5cm}| m{2.5cm}| m{8.53cm} | }
\hline
$a_7$ & $b_7$ & $\sigma_7$ & non-monogenic $n$ \\ 
\hline
0 & 1 & \begin{center} \Scale[.9]{$\begin{bmatrix} 
0 & 1 \\
-7 & 0
\end{bmatrix}$} \end{center} &  
4, 8, {\color{red}12,} 16, 24, 43, 48, 75, 80, 86, 96, 100, 120, 150, 160, 172, 191,
200, 240, 300, 344, 382, 400, 480, {\color{red}516,} 600, 684, 688, 764, 774, 800,
817, {\color{red}911,} 912\\ [.1 ex]
\hline
0 & 2 & \begin{center} \Scale[.9]{$\begin{bmatrix} 
1 & 2 \\
-4 & -1
\end{bmatrix}$} \end{center} &  
4, 8, {\color{red}12,} 16, 24, 43, 48, 75, 80, 86, 96, 100, 120, 150, 160, 172, 191,
200, 240, 300, 344, 382, 400, 480, {\color{red}516,} 600, 684, 688, 764, 774, 800,
817, {\color{red}911,} 912
\\ [.1 ex]
\hline
1 & 1 & \begin{center} \Scale[.9]{$\begin{bmatrix} 
1 & 1 \\
-7 & 0
\end{bmatrix}$} \end{center} &  \\ [.1 ex]
\hline
$-1$  & 1 & \begin{center}\Scale[.9]{$\begin{bmatrix} 
0 & 1 \\
-7 & -1
\end{bmatrix}$} \end{center} &   \\ [.1 ex]
\hline
1  & 3 & \begin{center} \Scale[.9]{$\begin{bmatrix} 
2 & 3 \\
-3 & -1
\end{bmatrix}$} \end{center} & {\color{red} 3,} 36, 936  \\ [.1 ex]
\hline
 $-1$  & 3 & \begin{center} \Scale[.9]{$\begin{bmatrix} 
1 & 3 \\
-3 & -2
\end{bmatrix}$} \end{center} & 3, 9, 18, 36, 936 \\ [.1 ex]
\hline
 $2$  & 1 & \begin{center} \Scale[.9]{$\begin{bmatrix} 
1 & 1 \\
-6 & 1 
\end{bmatrix}$} \end{center} & {\color{red}10,} 20 \\ [.1 ex]
\hline
 $-2$  & 1 & \begin{center} \Scale[.9]{$\begin{bmatrix} 
-1 & 1 \\
-6 & -1 
\end{bmatrix}$} \end{center} & {\color{red}10,} 20 \\ [.1 ex]
\hline
 $3$  & 1 & \begin{center} \Scale[.9]{$\begin{bmatrix} 
2 & 1 \\
-5 & 1
\end{bmatrix}$} \end{center} & 15 \\ [.1 ex]
\hline
 $-3$  & 1 & \begin{center} \Scale[.9]{$\begin{bmatrix} 
-1 & 1 \\
-5 & -2
\end{bmatrix}$} \end{center} & 15 \\ [.1 ex]
\hline
 $4$  & 1 & \begin{center} \Scale[.9]{$\begin{bmatrix} 
2 & 1 \\
-3 & 2
\end{bmatrix}$} \end{center} & 4, 36, 936 \\ [.1 ex]
\hline
 $-4$  & 1 & \begin{center} \Scale[.9]{$\begin{bmatrix} 
-2 & 1 \\
-3 & -2
\end{bmatrix}$} \end{center} & 4, 9, 36, 936 \\ [.1 ex]
\hline
 $4$  & 2 & \begin{center} \Scale[.9]{$\begin{bmatrix} 
3 & 2 \\
-2 & 1
\end{bmatrix}$} \end{center} & 4, 36, 936 \\ [.1 ex]
\hline
 $-4$  & 2 & \begin{center} \Scale[.9]{$\begin{bmatrix} 
-1 & 2 \\
-2 & -3
\end{bmatrix}$} \end{center} & 4, 9, 18, 36, 936 \\ [.1 ex]
\hline
 $5$  & 1 & \begin{center} \Scale[.9]{$\begin{bmatrix} 
3 & 1 \\
-1 & 2 
\end{bmatrix}$} \end{center} & 9, 18, 36, 936 \\ [.1 ex]
\hline
 $-5$  & 1 & \begin{center} \Scale[.9]{$\begin{bmatrix} 
-2 & 1 \\
-1 & -3 
\end{bmatrix}$} \end{center} & 36, 936 \\ [.1 ex]
\hline
\end{tabular}
\end{adjustbox}
\vspace{.05 in}
\caption{Using the splitting of 7 in $\QQ(E[n])$ to show non-monogeneity for $n<1000$}
\label{Frobenius7}
\end{table}

\newpage

\begin{table}[h!]
\centering
\begin{adjustbox}{width=4.1 in,center}
\begin{tabular}{ | m{.5cm} | m{.5cm}| m{2.5cm}| m{6.5cm} | }
\hline
$a_{11}$ & $b_{11}$ & $\sigma_{11}$ & non-monogenic $n$ \\
\hline
0 & 1 & \begin{center} \Scale[.9]{$\begin{bmatrix} 
0 & 1\\ 
-11 & 0 
\end{bmatrix}$} \end{center} &   
6, 12, {\color{red}15,} 20, 24, 30, {\color{red}37,} 40, 60, 74, 111, 120, 148, 183, 222, 240,
244, 305, 333, 366, 444, 488, 610, 666, 732, 915, 976
\\ [.1 ex]
\hline
0 & 2 & \begin{center} \Scale[.9]{$\begin{bmatrix} 
1 & 2\\ 
-6 & -1 
\end{bmatrix}$} \end{center} &   
6, 12, {\color{red}15,} 20, 24, 30, {\color{red}37,} 40, 60, 74, 111, 120, 148, 183, 222, 240,
244, 305, 333, 366, 444, 488, 610, 666, 732, 915, 976\\ [.1 ex]
\hline
1 & 1 & \begin{center} \Scale[.9]{$\begin{bmatrix} 
1 & 1\\ 
-11 & 0 
\end{bmatrix}$} \end{center} &   \\ [.1 ex]
\hline
$-1$  & 1 & \begin{center}\Scale[.9]{$\begin{bmatrix} 
0 & 1 \\
-11 & -1 
\end{bmatrix}$} \end{center} & $10$  \\ [.1 ex]
\hline
2  & 1 & \begin{center} \Scale[.9]{$\begin{bmatrix} 
1 & 1 \\
-1 & 1 
\end{bmatrix}$} \end{center} &   \\ [.1 ex]
\hline
 $-2$  & 1 & \begin{center} \Scale[.9]{$\begin{bmatrix} 
-1 & 1 \\
-10 & -1 
\end{bmatrix}$} \end{center} & {\color{red}7} \\ [.1 ex]
\hline
 $3$  & 1 & \begin{center} \Scale[.9]{$\begin{bmatrix} 
2 & 1 \\
-9 & 1 
\end{bmatrix}$} \end{center} &  \\ [.1 ex]
\hline
 $-3$  & 1 & \begin{center} \Scale[.9]{$\begin{bmatrix} 
-1 & 1 \\
-9 & -2 
\end{bmatrix}$} \end{center} &  \\ [.1 ex]
\hline
 $4$  & 1 & \begin{center} \Scale[.9]{$\begin{bmatrix} 
2 & 1 \\
-7 & 2 
\end{bmatrix}$} \end{center} &  \\ [.1 ex]
\hline
 $-4$  & 1 & \begin{center} \Scale[.9]{$\begin{bmatrix} 
-2 & 1 \\
-7 & -2 
\end{bmatrix}$} \end{center} &  \\ [.1 ex]
\hline
 $4$  & 2 & \begin{center} \Scale[.9]{$\begin{bmatrix} 
3 & 2 \\
-4 & 1 
\end{bmatrix}$} \end{center} & 8, 10{\color{red}, 16} \\ [.1 ex]
\hline
 $-4$  & 2 & \begin{center} \Scale[.9]{$\begin{bmatrix} 
-1 & 2 \\
-4 & -3 
\end{bmatrix}$} \end{center} & 8{\color{red}, 16, 86} \\ [.1 ex]
\hline
 $5$  & 1 & \begin{center} \Scale[.9]{$\begin{bmatrix} 
3 & 1 \\
-5 & 2 
\end{bmatrix}$} \end{center} & {\color{red}7,} 14 \\ [.1 ex]
\hline
 $-5$  & 1 & \begin{center} \Scale[.9]{$\begin{bmatrix} 
-2 & 1 \\
-5 & -3 
\end{bmatrix}$} \end{center} &  \\ [.1 ex]
\hline
 $6$  & 1 & \begin{center} \Scale[.9]{$\begin{bmatrix} 
3 & 1 \\
-2 & 3 
\end{bmatrix}$} \end{center} & 6, {\color{red}45,} 90 \\ [.1 ex]
\hline
 $-6$  & 1 & \begin{center} \Scale[.9]{$\begin{bmatrix} 
-3 & 1 \\
-2 & -3 
\end{bmatrix}$} \end{center} & 6, {\color{red}45,} 90 \\ [.1 ex]
\hline
\end{tabular}
\end{adjustbox}
\vspace{.05 in}
\caption{Using the splitting of 11 in $\QQ(E[n])$ to show non-monogeneity for $n<1000$.}
\label{Frobenius11}
\end{table}

\newpage

\bibliography{Bibliography}
\bibliographystyle{abbrvnat}

\end{document}